\documentclass[reqno,12pt]{amsart}

\usepackage{graphicx}
\usepackage{subfigure}
\usepackage{enumerate}
\usepackage{xspace}
\usepackage{microtype}
\usepackage{lmodern}
\usepackage{url}

\usepackage{color}

\newcommand{\df}[1]{\ensuremath{\operatorname{df}({#1})}\xspace}
\newcommand{\core}{{\mathrm{core}}}
\newcommand{\girth}{{\mathrm{girth}}}
\newcommand{\den}{\operatorname{dn}\xspace}

\theoremstyle{plain}
\newtheorem{theorem}{Theorem}[section]
\newtheorem{lemma}[theorem]{Lemma}
\newtheorem{proposition}[theorem]{Proposition}
\newtheorem{corollary}[theorem]{Corollary}

\theoremstyle{definition}

\newtheorem*{example*}{Example}

\title{Girth, oddness, and colouring defect of snarks}
\author[J. Karab\'a\v{s}]{J\'an Karab\'a\v{s}}
\email[J. Karab\'a\v{s}]{karabas@savbb.sk}
\author[E. M\'a\v cajov\'a]{Edita M\'a\v cajov\'a}
\email[E. M\'a\v cajov\'a]{macajova@dcs.fmph.uniba.sk}
\author[R. Nedela]{Roman Nedela}
\email[R. Nedela]{nedela@savbb.sk}
\author[M. \v Skoviera]{Martin \v Skoviera}
\email[M. \v Skoviera]{skoviera@dcs.fmph.uniba.sk}

\address[J. Karab\'a\v{s}]{
Department of Computer Science, Faculty of Natural Sciences,
Matej Bel University, Bansk\'a Bystrica, Slovakia
}

\address[E. M\'a\v cajov\'a, M. \v Skoviera]{
Comenius University, Mlynsk\' a dolina, Bratislava, Slovakia
}

\address[R. Nedela]{
Faculty of Applied Sciences, University of West Bohemia,
Pilsen, Czech Republic}

\address[J. Karab\'a\v{s}, R. Nedela]{
Mathematical Institute of Slovak Academy of Sciences,
Bansk\'a Bystrica, Slovakia
}

\begin{document}

\begin{abstract}
The colouring defect of a cubic graph, introduced by Steffen in
2015, is the minimum number of edges that are left uncovered by
any set of three perfect matchings. Since a cubic graph has
defect $0$ if and only if it is $3$-edge-colourable, this
invariant can measure how much a cubic graph differs from a
$3$-edge-colourable graph. Our aim is to examine the
relationship of colouring defect to oddness, an extensively
studied measure of uncolourability of cubic graphs, defined as
the smallest number of odd circuits in a $2$-factor. We show
that there exist cyclically $5$-edge-connected snarks (cubic
graphs with no $3$-edge-colouring) of oddness $2$ and arbitrarily
large colouring defect. This result is achieved by means of a
construction of cyclically $5$-edge-connected snarks with
oddness $2$ and arbitrarily large girth. The fact that our
graphs are cyclically $5$-edge-connected significantly
strengthens a similar result of Jin and Steffen (2017), which
only guarantees graphs with cyclic connectivity at most $3$. At
the same time, our result improves Kochol's original
construction of snarks with large girth (1996) in that it
provides infinitely many nontrivial snarks of any prescribed
girth $g\ge 5$, not just girth at least~$g$.
\end{abstract}

\maketitle
\thispagestyle{empty}
\section{Introduction}
\noindent{}The \emph{colouring defect} of a cubic graph $G$,
denoted by $\df{G}$, is the smallest number of edges left
uncovered by any set of three perfect matchings of $G$. For
brevity, we usually drop the adjective ``colouring'' and speak
of the \emph{defect} of a cubic graph. Clearly, the defect of a
$3$-edge-colourable cubic graph is $0$, but in general it can
be arbitrarily large. Defect thus can be regarded as a measure
of uncolourability of a cubic graph.

The concept of defect was introduced by Steffen as $\mu_3(G)$
in \cite{S2}, where he also established its fundamental
properties. Among other things he proved that every
$2$-connected cubic graph which is not $3$-edge-colourable -- a
snark -- has defect at least three. Another notable result of
\cite{S2} states that the defect of a snark is at least as
large as one half of its girth. Since there exist snarks of
arbitrarily large girth \cite{Ko}, there exist snarks of
arbitrarily large defect.

The defect of a cubic graph was further examined by Jin and
Steffen in \cite{JS} and was also discussed in the survey of
uncolourability measures by Fiol et al.
\cite[pp.\,13--14]{FMS-survey}. Jin and Steffen~\cite{JS} studied
the relationship of defect to other measures of
uncolourability, in particular its relationship to oddness. The
\emph{oddness} of a cubic graph~$G$, denoted by $\omega(G)$, is
the minimum number of odd circuits in a $2$-factor of $G$; it
is correctly defined for any bridgeless cubic graph. In
\cite[Corollary 2.4]{JS}, Jin and Steffen proved  that $\df
G\geq 3\omega(G)/2$ and investigated the extremal case where
$\df G= 3\omega(G)/2$ in detail. The inequality implies that
with increasing oddness the difference between defect and
oddness becomes arbitrarily large.

Measures of uncolourability are particularly interesting for
\emph{nontrivial snarks}, those which are cyclically
$4$-edge-connected and have girth at least $5$. The reason is
that several important conjectures in graph theory, such as the
cycle double cover conjecture, Fulkerson's conjecture, and
others, would have their minimal counterexamples in this class
\cite{J85, MM}. Note that nontrivial snarks with arbitrarily
large oddness were constructed in~\cite{Ko2, LMMS, S4}. In this
context it is natural to ask whether the difference between
defect and oddness remains arbitrarily large when oddness is
fixed even for nontrivial snarks. To this end, Jin and
Steffen~\cite[Theorem~3.4]{JS} proved that for any given
oddness $\omega>0$ and any $d\ge3\omega/2$ there exists a
bridgeless cubic graph with oddness $\omega$ and defect at
least~$d$. However, their construction produces graphs with
cyclic connectivity not exceeding~$3$.

Our main result, Theorem~\ref{thm:snarxlargegirth}, improves
the result of Jin and Steffen for $\omega=2$ by establishing
the existence of cyclically $5$-edge-connected snarks with
oddness $2$ and arbitrarily large defect. This result is
achieved through a construction of cyclically
$5$-edge-connected snarks with oddness $2$ and arbitrarily
large girth. Our construction strengthens the original
construction by Kochol~\cite{Ko} in that it provides infinitely
many nontrivial snarks of \emph{any} prescribed girth $g\ge 6$,
not just girth at least $g$. Snarks with arbitrarily large
defect, oddness~$2$, and cyclic connectivity $4$ can be
constructed in a similar manner. Note that the existence of
nontrivial snarks of arbitrarily large girth was not confirmed
until 1996, when Kochol~\cite{Ko} disproved a conjecture by
Jaeger and Swart~\cite[Conjecture 2]{JSw}.

\medskip

A detailed study of colouring defect, focused on snarks with
defect $3$, is carried out in our companion
papers~\cite{KMNS-defect3, KMNS-pmi}, in which the basic
properties of defect and structures related to it are discussed
in a greater detail.

\section{Preliminaries}

\noindent{}All graphs in this paper are finite and for the most
part cubic (3-valent). Multiple edges and loops are permitted.
We use the term \emph{circuit} to mean a connected $2$-regular
graph. The length of a shortest circuit in a graph is its
\emph{girth}. By a $k$-\emph{cycle} we mean a circuit of
length~$k$.

A graph $G$ is said to be \emph{cyclically $k$-edge-connected}
if the removal of fewer than $k$ edges from $G$ cannot create a
graph with at least two components containing circuits. An edge
cut $S$ in $G$ that separates two circuits from each other is
\emph{cycle-separating}. It is not difficult to see that the
set of edges of a cubic graph leaving a shortest circuit is
cycle-separating in all connected cubic graphs other than the
complete bipartite graph $K_{3,3}$, the complete graph $K_4$,
and the graph consisting of two vertices and three parallel
edges. An edge cut of a cubic graph consisting of independent
edges is always cycle-separating. Conversely, a
cycle-separating edge cut of minimum size is independent. A
cycle-separating edge cut that separates a shortest cycle from
the rest of $G$ is called \emph{trivial}.

Large graphs are typically constructed from smaller building
blocks called multipoles. Similarly to graphs, each
\emph{multipole} $M$ has its vertex set $V(M)$, its edge set
$E(M)$, and an incidence relation between vertices and edges.
Each edge of $M$ has two ends, and each end may, but need not
be, incident with a vertex of $M$. An end of an edge that is
not incident with a vertex is called a \emph{free end} or a
\emph{semiedge}. An edge with exactly one free end is called a
\emph{dangling edge}. An \emph{isolated edge} is an edge whose
both ends are free. All multipoles considered in this paper are
\emph{cubic}; it means that every vertex is incident with
exactly three edge-ends. An \emph{$n$-pole} $M$ is a multipole
with $n$ free ends. If its free ends are $s_1,s_2,\ldots, s_n$,
we write $M=M(s_1,s_2,\ldots, s_n)$.

Free ends of a multipole can be distributed into pairwise
disjoint sets, called \emph{connectors}. Connectors of
multipoles are usually matched and their free ends are
subsequently identified in a straightforward manner to produce
cubic graphs. An \emph{$(n_1, n_2,\ldots, n_k)$-pole} is an
$n$-pole with $n=n_1+n_2+\cdots + n_k$ whose semiedges are
distributed into $k$ connectors $S_1,S_2,\ldots, S_k$, each
$S_i$ being of size $n_i$. A \emph{dipole} is a multipole with
two connectors, while a \emph{tripole} is a multipole with
three connectors. An \emph{ordered} multipole is the one where
each connector is endowed with a linear order.

An \emph{edge colouring} of a multipole $M$ is a mapping from
the edge set of $M$ to a set of colours such that any two
edge-ends incident with the same vertex carry distinct colours.
A \emph{$k$-edge-colouring} is a colouring where the set of
colours has $k$ elements. A cubic graph $G$ is
\emph{colourable} if it admits a $3$-edge-colouring. A
$2$-connected cubic graph which does not admit a
$3$-edge-colouring is called a \emph{snark}.

In the study of snarks it is useful to take the colours $1$,
$2$, and $3$ to be the nonzero elements of the group
$\mathbb{Z}_2\times\mathbb{Z}_2$. Specifically, one can
identify a colour with its binary representation: $1=(0,1)$,
$2=(1,0)$, and $3=(1,1)$. The condition that the three colours
meeting at any vertex $v$ are all distinct then becomes
equivalent to requiring that the sum of colours at $v$ is
$0=(0,0)$. In other words, identifying the colours with the
elements  of $\mathbb{Z}_2 \times \mathbb{Z}_2-\{0\}$ turns a
$3$-edge-colouring to a nowhere-zero $\mathbb{Z}_2 \times
\mathbb{Z}_2$-flow. Recall that an \textit{$A$-flow} on a
multipole $M$ is a function $\sigma\colon E(M)\to A$, with
values in an abelian group~$A$, together with an orientation of
$M$, such that \emph{Kirchhoff's law} is fulfilled: at each
vertex of $M$ the sum of all incoming values equals the sum of
all outgoing ones. An $A$-flow $\sigma$ is \emph{nowhere-zero}
if $\sigma(e)\ne 0$ for each edge $e$ of $G$. If $x=-x$ for
every $x\in A$, the orientation of $M$ can be ignored. This is
possible precisely when $A$ is isomorphic to an elementary
abelian $2$-group $\mathbb{Z}_2^n$.

The following well-known statement is a direct consequence of
Kirchhoff's law. Rougly speaking, it tells us that the total
outflow from any nonempty set of vertices equals $0$.

\begin{lemma}[Parity Lemma]\label{lem:par}
Let $M = M(s_1,s_2, \ldots,s_n)$ be a $n$-pole endowed with a
$3$-edge-colouring $\sigma$. Then
$$\sum_{i=1}^n\sigma(s_i)=0.$$
Equivalently, the number of free ends of $M$ carrying any fixed
colour has the same parity as $n$.
\end{lemma}

Our definition of a snark leaves the concept as wide as
possible since more restrictive definitions could lead to
overlooking certain important phenomena that occur among
snarks. In this manner we follow works of Cameron et al.
\cite{CCW}, Nedela and \v Skoviera \cite{NS-decred}, Steffen
\cite{S1}, and others, rather than a common approach where
snarks are required to be cyclically $4$-edge-connected and
have girth at least $5$, see for example~\cite{FMS-survey}. In
this paper, such snarks are called \emph{nontrivial}. The
problem of nontriviality of snarks has been widely discussed in
the literature, see for example \cite{CCW,NS-decred,S1}. Here
we adopt a systematic approach to nontriviality of snarks
proposed by Nedela and \v Skoviera \cite{NS-decred} based on
the concept of removability of certain sets of vertices or
subgraphs. We say that an induced subgraph $H$ of a snark $G$
is \emph{non-removable} if $G - V(H)$ is colourable; otherwise,
$H$ is \emph{removable}. It is an easy consequence of Parity
Lemma that circuits of length at most $4$ in snarks are
removable.

\section{Arrays of perfect matchings and the defect of a snark}

\noindent{}In order to formalise our discussion of colouring
defect it is convenient to define a \emph{$3$-array of perfect
matchings} in a cubic graph $G$, briefly a \emph{$3$-array} of
$G$, as an arbitrary collection $\mathcal{M}=\{M_1, M_2, M_3\}$
of three not necessarily distinct perfect matchings of $G$.
Since every proper $3$-edge-colouring can be regarded as an
array whose members are the three colour classes, $3$-arrays
can be viewed as approximations of $3$-edge-colourings. An edge
of $G$ that belongs to at least one of the perfect matchings of
the array $\mathcal{M}=\{M_1, M_2, M_3\}$ will be considered to
be \emph{covered}. An edge will be called \emph{uncovered},
\emph{simply covered}, \emph{doubly covered}, or \emph{triply
covered} if it belongs, respectively, to zero, one, two, or
three members of~$\mathcal{M}$.

Given a graph $G$, it is a natural task to maximise the number
of covered edges in a $3$-array of $G$, or equivalently, to
minimise the number of uncovered ones. A $3$-array that leaves
the minimum number of uncovered edges will be called
\emph{optimal}. The number of edges left uncovered by an
optimal $3$-array is the \emph{colouring defect} of $G$,
denoted by $\df{G}$.

Let $\mathcal{M}=\{M_1, M_2, M_3\}$ be a $3$-array of a cubic
graph $G$. One way to describe $\mathcal{M}$ is based on
regarding the indices $1$, $2$, and $3$ as colours. Since the
same edge may belong to more than one member of $\mathcal{M}$,
an edge of $G$ may receive from $\mathcal{M}$ more than one
colour. To each edge $e$ of $G$ we can therefore assign the
list $\varphi(e)$ of all colours in lexicographic order it
receives from $\mathcal{M}$. In this way $\mathcal{M}$ gives
rise to a mapping
$$\varphi\colon E(G)\to\{\emptyset, 1, 2, 3, 12, 13, 23, 123\}$$
where $\emptyset$ denotes the empty list. Such a mapping
determines a $3$-array of $G$ if and only if each number from
$\{1,2,3\}$ occurs precisely once on the edges around any
vertex. Moreover, $\varphi$ is a proper edge colouring if and
only if $G$ has no triply covered edge with respect to
$\mathcal{M}$. For more details, see~\cite{KMNS-defect3}.

Another important structure associated with a $3$-array is its
core. The \emph{core} of a $3$-array $\mathcal{M}=\{M_1, M_2,
M_3\}$ of $G$ is the subgraph of $G$ induced by all the edges
of $G$ that are not simply covered; we denote it by
$\core(\mathcal{M})$. The core will be called \emph{optimal}
whenever $\mathcal{M}$ is optimal. Given a $3$-array
$\mathcal{M}$, let $E_i=E_i(\mathcal{M})$ denote the set of all
edges of $G$ that belong to precisely $i$ perfect matchings of
$\mathcal{M}$, where $0\le i\le 3$. The edge set of
$\core(\mathcal{M})$ thus coincides with $E_0(\mathcal{M})\cup
E_{2}(\mathcal{M})\cup E_{3}(\mathcal{M})$. It is worth
mentioning that if $G$ is $3$-edge-colourable and $\mathcal{M}$
consists of three disjoint perfect matchings, then
$\core(\mathcal{M})$ is empty.
If $G$ is not $3$-edge-colourable, then the
core must be nonempty for every $3$-array $\mathcal{M}$ of $G$.

Figure~\ref{fig:petersen_core} shows the Petersen graph endowed
with a $3$-array whose core is the ``outer'' $6$-cycle. The
core is in fact optimal.

\begin{figure}[h!]
 \centering
 \includegraphics[scale=1.2]{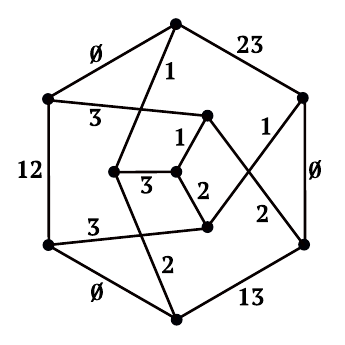}
 \caption{An optimal $3$-array of the Petersen graph}
\label{fig:petersen_core}
\end{figure}

The following proposition, much of which was proved by Steffen
in {\cite[Lemma~2.2]{S2}} and \cite[Lemma~2.1]{JS} lists the
most fundamental properties of cores.

\begin{proposition}\label{prop:core}
Let $\mathcal{M}=\{M_1, M_2, M_3\}$ be an arbitrary $3$-array
of perfect matchings of a snark $G$. Then the following hold:
\begin{enumerate}[{\rm (i)}]
\item Every component of $\core(\mathcal{M})$ is either an
    even circuit or a subdivision of a cubic graph. If $G$
    has no triply covered edge, then $\core(\mathcal{M})$
    is a set of disjoint even circuits, and vice-versa.
\item Every $2$-valent vertex of $\core(\mathcal{M})$ is
    incident with one doubly covered edge and one uncovered
    edge, while every $3$-valent vertex is incident with
    one triply covered edge and two uncovered edges.
\item $|E_0(\mathcal{M})| =
    |E_2(\mathcal{M})|+2|E_3(\mathcal{M})|$.
\item $G-E_0(\mathcal{M})$ is $3$-edge-colourable.
\end{enumerate}
\end{proposition}

Proposition~\ref{prop:core}~(i) implies that the smallest
possible cores are the $2$-cycle and the $4$-cycle. However,
Parity Lemma implies that circuits of length at most four are
removable, so neither of them can occur as a core.
Consequently, the following important fact holds.

\begin{corollary}[{\cite{S2}}]
The defect of every snark has value at least three.
\end{corollary}

Following Steffen \cite{S2} we say that the core of a $3$-array
$\mathcal{M}$  of a cubic graph $G$ is \emph{cyclic} if each
component of $\core(\mathcal{M})$ is a circuit. By
Proposition~\ref{prop:core}~(ii), the core is cyclic if and
only if $G$ has no triply covered edge. The well-known
conjecture  of Fan and Raspaud \cite{FR} suggests that every
bridgeless cubic graph has three perfect matchings $M_1$,
$M_2$, and $M_3$ with $M_1\cap M_2\cap M_3=\emptyset$.
Equivalently, the conjecture states that  every bridgeless
cubic graph has a $3$-array with a cyclic core. The conjecture
is trivially true for $3$-edge-colourable graphs. M\'a\v
cajov\'a and \v Skoviera \cite{MS-comb} proved this conjecture
to be true for cubic graphs with oddness $2$. We emphasise that
neither the conjecture nor the proved facts suggest anything
about optimal cores.

\section{Oddness, girth and colouring defect}\label{sec:otherinv}

\noindent{}In this section we discuss relationships between
several measures of uncolourabilty of cubic graphs (in the
sense of the survey \cite{FMS-survey}), with particular
emphasis on oddness and defect. Most of the inequalities proved
here are known, however, the proofs which we offer are cleaner
and more transparent. The main result of this paper,
Theorem~\ref{thm:snarxlargegirth} (to be proved in the next
section), relates oddness, defect and -- implicitly -- girth.
Its proof uses one of the inequalities established in present
section.

Let $G$ be a bridgeless cubic graph. The \emph{resistance} of
$G$, denoted by $\rho(G)$, is the smallest number of edges
whose removal from $G$ yields a $3$-edge-colourable graph. It
is well known that $\rho(G)\le\omega(G)$ and that $\rho(G)=2$
if and only if $\omega(G)=2$, see~\cite[Lemma~2.5]{S1}. The
\emph{density} $\den(G)$ of $G$ is the minimum number of common
edges that two perfect matchings in $G$ can have. This
invariant was introduced by Steffen in~\cite{S3} and denoted by
$\gamma_2(G)$ in~\cite{JS}.  Jin and Steffen
in~\cite[Theorem~2.2]{JS} proved that
\begin{equation}\label{eq:steffen}
\omega(G)\leq 2\den(G)\leq\df{G}-1,
\end{equation}
if $G$ is not $3$-edge-colourable. As a consequence, if
$\df{G}=3$, then $\den(G)=1$ and $\omega(G)=2$.

We prove \eqref{eq:steffen} starting with the inequality on the
left-hand side.

\begin{proposition}\label{lem:omegagamma}
If $G$ is a bridgeless cubic graph, then $\omega(G)\leq
2\den(G)$.
\end{proposition}

\begin{proof}
Let $M_1$ and $M_2$ be any two perfect matchings of $G$. Take
the $2$-factor $F_1$ complementary to $M_1$, and assume that it
has $c$ odd circuits. Since every set with an odd number of
vertices sends out an edge of $M_2$, each odd circuit of $F_1$
is incident with at least one edge from $M_1\cap M_2$. If $E'$
denotes the set of all edges of $M_1\cap M_2$ incident with an
odd circuit of $F_1$, then clearly $|E'|\geq c/2\geq \omega/2$,
where $\omega=\omega(G)$. Consequently, $|M_1\cap M_2|\geq|E'|
\geq \omega/2$ for each pair $M_1$ and $M_2$ of perfect
matchings of $G$, and so $\den(G) = \min_{M_1,M_2}|M_1\cap M_2|
\geq \omega/2$.
\end{proof}

Now we are ready for the inequality on the right-hand side of
\eqref{eq:steffen}.

\begin{proposition} If $G$ is a snark, then
\[
\df{G}\geq 2\den(G)+1.
\]
\end{proposition}
\begin{proof}
Let $\mathcal{M}=\{M_1, M_2, M_3\}$ be an optimal $3$-array of
$G$. Since $G$ is a snark, $\core(\mathcal{M})$ is nonempty. We
claim that $\core(\mathcal{M})$ contains at least one doubly
covered edge. Suppose not. Then $\core(\mathcal{M})$ consists
of uncovered and triply covered edges, which implies that
$M_1=M_2=M_3$. Pick an uncovered edge $e$ and take a perfect
matching $M_1'$ containing~$e$; it is well known that such a
perfect matching always exists \cite{Plesnik}. Clearly, the
$3$-array $\{M_1',M_2,M_3\}$ has fewer uncovered edges than
$\mathcal{M}$, so $\mathcal{M}$ was not optimal. Thus, if
$\mathcal{M}$ is optimal, there exist indices $i\ne j$ such
that $|M_i\cap M_j|-|E_3|\geq 1$; without loss of generality we
may assume that $|M_1\cap M_2|-|E_3|\geq 1$. By applying
Proposition~\ref{prop:core}~(iii) we obtain
\begin{align*}
\df{G}&=|E_2|+2|E_3|=\left(\sum_{i\neq j}|M_i\cap M_j|\right)-|E_3|
\geq |M_1\cap M_3| + |M_2\cap M_3| + 1\\ &\geq 2\den(G)+1,
\end{align*}
as required.
\end{proof}

Proposition~\ref{prop:core}~(iv) implies that
$\df{G}\ge\rho(G)$ for every bridgeless cubic graph $G$. Jin
and Steffen \cite[Corollary~2.4]{JS} proved the following
stronger result.

\begin{theorem}\label{thm:oddness}
For every bridgeless cubic graph $G$ one has
$$\df{G}\geq 3\omega(G)/2.$$
\end{theorem}

\begin{proof}
Let $\mathcal{M} = \{M_1, M_2, M_3\}$ be an optimal $3$-array
of $G$, and for $i\in\{1,2,3\}$ let $F_i$ be the $2$-factor
$F_i$ complementary to $M_i$. Our aim is to estimate the number
of odd circuits in each $F_i$ and then use the estimate to
bound the oddness of $G$.

For each $i\in\{1,2,3\}$ we partition the set of odd circuits
of $F_i$ into three subsets $\mathcal{C}_i^1$,
$\mathcal{C}_i^2$, and $\mathcal{C}_i^3$ as follows:
\begin{itemize}
\item[(i)] $\mathcal{C}_i^1$ will consist of all odd
    circuits of $F_i$ contained in $\core(\mathcal{M})$ in
    which all edges are uncovered;
\item[(ii)] $\mathcal{C}_i^2$ will consist of all odd
    circuits of $F_i$ contained in $\core(\mathcal{M})$
    which contain at least one doubly covered edge; and
\item[(iii)] $\mathcal{C}_i^3$ will consist of all odd
    circuits not contained in $\core(\mathcal{M})$.
\end{itemize}

Observe that the edges leaving a circuit $C$ from
$\mathcal{C}_i^1$ are all triply covered. In other words,
$\mathcal{C}_1^1=\mathcal{C}_2^1=\mathcal{C}_3^1$, so for
simplicity we write $\mathcal{C}_i^1=\mathcal{C}^1$. A vertex
of $G$ incident with a triply covered edge will be called
\textit{special}. Next, each circuit $C$ from any
$\mathcal{C}_i^2$ consists of uncovered edges and doubly
covered edges, and since $C$ is odd, at least two uncovered
edges of $C$ must be adjacent. It follows that each
$C\in\mathcal{C}_i^2$ has at least one special vertex.

At first we derive a bound on $|\mathcal{C}_i^3|$. Pick an
arbitrary circuit $C\in\mathcal{C}_i^3$; since $C$ is odd, it
contains an edge of $\core(\mathcal{M})$. Consider a component
of the intersection of circuit $C$ and $\core(\mathcal{M})$,
which must be a path $P\subseteq C$. Let $u$ and $v$ be the
endvertices of $P$, and let $e$ and $f$ be the edges of $M_i$
incident with $u$ and $v$, respectively. By
Proposition~\ref{prop:girth}, $e$ and $f$ cannot be simply
covered, because each of them is adjacent to an edge of
$\core(\mathcal{M})\cap C$ and to a simply covered edge of $C$.
As both $e$ and $f$ are covered, they must be doubly covered
and hence belong to $E_2\cap M_i$. In other words, each $C\in
\mathcal{C}_i^3$  produces at least two edges from $E_2\cap
M_i$.

Form an auxiliary graph $X_i$ with bipartition
$\{\mathcal{C}_i^3,E_2\cap M_i\}$, where $C\in {\mathcal
C}_i^3$ is joined to $e\in E_2\cap M_i$ whenever $e$ is
incident with $C$. As previously explained, $\deg(C)\geq 2$ for
each $C\in\mathcal{C}_i^3$ while $\deg(e)\leq 2$ for each $e\in
E_2\cap M_i$. Hence, counting the edges of $X_i$ in two ways
yields
\[
2|\mathcal{C}_i^3|\le\sum_{C\in\mathcal{C}_i^3}\deg_{X_i}(C)=|E(X_i)|=
\sum_{e\in E_2\cap M_i}\deg_{X_i}(e)\le 2|E_2\cap M_i|.
\]
It follows that $|\mathcal{C}_i^3|\le |E_2\cap M_i|$ and hence
\begin{equation}\label{eq:c3}
\sum_{i=1}^3 |\mathcal{C}_i^3|\leq 2|E_2|.
\end{equation}

Now we bound $|\mathcal{C}^1|$. Let $S$ denote the set of
special vertices of $G$ with regard to $\mathcal{M}$. By
Proposition~\ref{prop:core}(ii), $|S|=2|E_3|$. Since each
vertex in a circuit from $\mathcal{C}^1$ is special, each
circuit from $\mathcal{C}^1$ has at least three special
vertices. Moreover, any two circuits from $\mathcal{C}^1$ are
disjoint. Therefore
\begin{equation}\label{eq:c1}
|\mathcal{C}^1|\leq |S|/3=2|E_3|/3.
\end{equation}

Lastly, we deal with $\mathcal{C}_i^2$. Every circuit from
$\mathcal{C}_i^2$ contains at least one pair of adjacent
uncovered edges, and therefore at least one special vertex. We
further show that any two circuits from
$\mathcal{C}_1^2\cup\mathcal{C}_2^2\cup\mathcal{C}_3^2$ are
disjoint. Suppose that this is false and two circuits
$C\in\mathcal{C}_i^2$ and $D\in\mathcal{C}_j^2$ have a nonempty
intersection. Then there exists an edge $e$ in $C\cap D$, which
is adjacent to two edges $f$ and $g$, such that $f$ lies in $C$
but not in $D$ and $g$ lies in $D$ but not in $C$. Since $f$ is
contained in $C$, it is uncovered or doubly covered. At the
same time, $f$ leaves~$D\in\mathcal{C}_j^2$, so it is simply or
triply covered, which is clearly impossible. Therefore $C\cap
D=\emptyset$. Hence
\begin{equation}\label{eq:c2}
\sum_{i=1}^3|\mathcal{C}_i^2|\leq |S|=2|E_3|.
\end{equation}

Summing up, if we denote the number of odd circuits in the
$2$-factor $F_i$ by $\omega_i$, from
\eqref{eq:c3}-\eqref{eq:c2} we obtain
\begin{equation}\label{eq:sum}
3\omega(G)\leq\sum_{i=1}^3\omega_i= 3|\mathcal{C}^1|+
\sum_{i=1}^3|\mathcal{C}_i^2|+
\sum_{i=1}^3|\mathcal{C}_i^3|\leq 4|E_3|+2|E_2|.
\end{equation}
By Proposition~\ref{prop:core}~(iii), the right-hand side of
\eqref{eq:sum} equals $2\df{G}$, and the theorem follows.
\end{proof}

The next result is due to Steffen~\cite[Corollary~2.5]{S2}.

\begin{proposition}\label{prop:girth}
For every snark $G$ one has $\df{G}\ge\lceil\girth(G)/2\rceil$.
\end{proposition}

\begin{proof}
Let $\mathcal{M}$ be a optimal $3$-array and let $H$ be its
core. Since each vertex of $H$ is either $2$-valent or
$3$-valent, $H$ contains a cycle $K$. Let $q$ be the length of
$K$. By Proposition~\ref{prop:core}~(ii), at least $\lceil
q/2\rceil$ edges of $K$ are left uncovered. Hence,
$$ \df{G}\ge \lceil q/2\rceil \ge \lceil\girth(G)/2\rceil,$$
as claimed.
\end{proof}

\section{Main result}
\noindent{}In this section we show that there exist nontrivial
snarks with oddness $2$ and arbitrarily large defect. As a
consequence, nontrivial snarks with oddness $2$ are split into
infinitely many subclasses according to their defect.

The proof makes use of the method of superposition, introduced
by Kochol in \cite{Ko}, whose main idea is to `inflate' a given
snark $G$ into a large cubic graph $\tilde G$ by substituting
vertices of $G$ with `fat vertices' (tripoles), called
\emph{supervertices}, and edges of $G$ with `fat edges'
(dipoles), called \emph{superedges}. Under suitable conditions
the inflated graph $\tilde G$ is a snark. For formal
definitions and a detailed description of the method we refer
the interested reader to the original paper \cite{Ko} or to a
recent paper \cite{MS-superp}. Our proof is self-contained.

\begin{theorem}\label{thm:snarxlargegirth}
There exist nontrivial snarks of oddness $2$ with arbitrarily
large defect.
\end{theorem}

\begin{proof}
To prove the theorem we modify the construction of snarks of
arbitrarily large girth due to Kochol \cite[Section~4]{Ko} in
such a way that a specified pair $\{u,v\}$ of adjacent vertices
of the resulting graph $\tilde G$ will be non-removable. This
fact will guarantee that $\omega(\tilde G)=2$ while $\df{\tilde
G}$ may take an arbitrarily large value, according to
Proposition~\ref{prop:girth}.

The key ingredient of our construction is a $(3,3)$-pole
$F=F_g$ that contains no cycles of length smaller than $g$ for
any prescribed~$g\ge 6$. It is represented in
Figure~\ref{fig:superedge} together with a partial
$3$-edge-colouring of its edges; the subgraphs indicated in
Figure~\ref{fig:superedge} as $M_g$ are copies of a $5$-pole
obtained from a suitable cubic graph $L_g$ of girth $g$ by
removing a path of length $2$. The $(3,3)$-pole $F_g$ will
serve as a superedge in our construction. It will be built in
several steps.

\begin{figure}[h]
\centering
\includegraphics[scale=1.2]{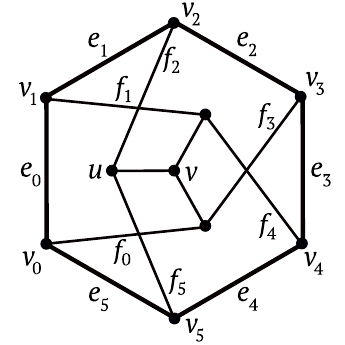}
\caption{The base graph for the superposition construction
in Theorem~\ref{thm:snarxlargegirth}}
\label{fig:core3}
\end{figure}
We start the construction of $F_g$ by taking three copies of
the Petersen graph, denoted by $P_1$, $P_2$, and $P_3$. In each
$P_i$ with $i\in\{1,3\}$ we choose a set $\{u_i,v_i,w_i\}$ of
three vertices at distance $2$ from each other, and in $P_2$ we
choose two edges $x_1x_2$ and $x_3x_4$ such that their
endvertices, if taken from distinct edges, are again at
distance~$2$ from each other. (For example, one can take
$\{u_i,v_i,w_i\}=\{v_0,v_2,v_4\}$, $x_1x_2=e_0$, and
$x_3x_4=e_3$, see Figure~\ref{fig:core3}.) It is important
that $\{u_i,v_i,w_i\}$, with $i\in\{1,3\}$, and
$\{x_1,x_2,x_3,x_4\}$ are \emph{decycling} sets, which
means that the removal of any of them from the Petersen graph
leaves an acyclic subgraph. We construct a new graph $K$ from
$P_1\cup (P_2-\{x_1x_2,x_3x_4\})\cup P_3$ as follows:
for $i\in\{1,3\}$ we create a new vertex $z_i$ by identifying
$v_i$ with $x_i$, and new vertex $z_{i+1}$ by identifying $w_i$
with $x_{i+1}$, thereby producing four $5$-valent vertices
$z_1$, $z_2$, $z_3$, and $z_4$. The result is shown in
Figure~\ref{fig:3pet5.30}. Set $W=\{z_1,z_2,z_3,z_4\}$ and
$U=W\cup\{u_1, u_3\}$. Note that $U$ is a decycling set
for~$K$.
\begin{figure}[h]
\centering
\subfigure[A graph $K$ with no nowhere-zero
$\mathbb{Z}_2\times\mathbb{Z}_2$-flow]{
{\includegraphics[width=0.75\textwidth]{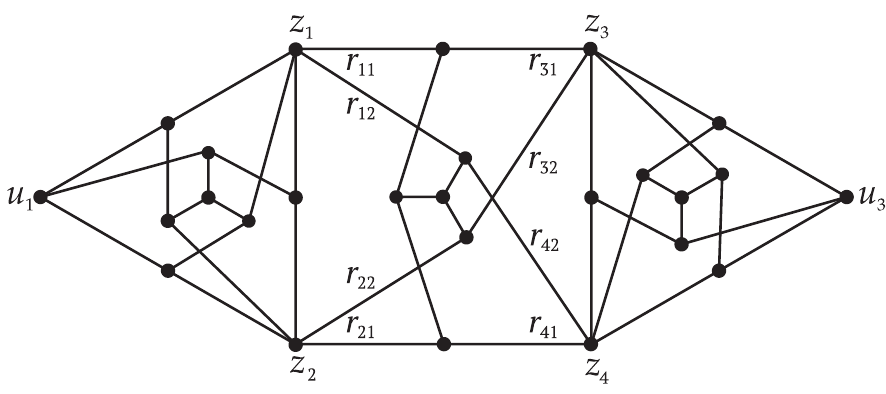}
\label{fig:3pet5.30}}
}
\subfigure[A proper $(3,3)$-pole $F_g$ of girth $g$]{
{\includegraphics[width=0.75\textwidth]{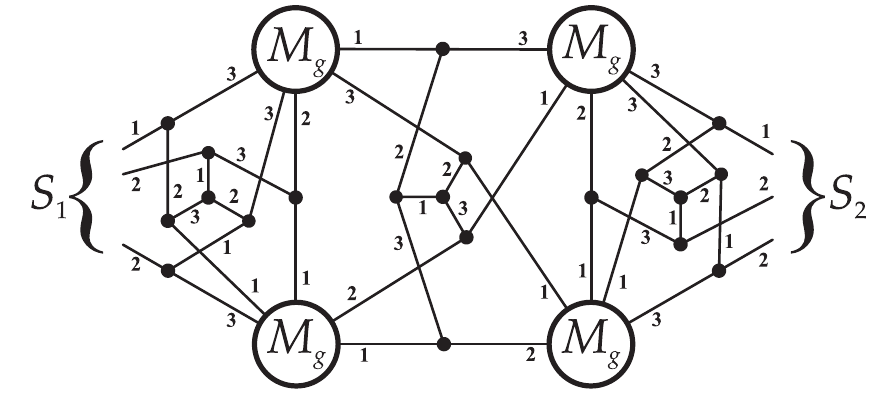}
\label{fig:superedge}}
}
\caption{Main ingredients of the construction of $\tilde G$.}
\end{figure}
We show that $K$ admits no nowhere-zero
$\mathbb{Z}_2\times\mathbb{Z}_2$-flow. Suppose to the contrary
that $\sigma$ is a nowhere-zero
$\mathbb{Z}_2\times\mathbb{Z}_2$-flow on $K$. For
$i\in\{1,2,3,4\}$ let $r_{i1}$ and $r_{i2}$ denote the edges
joining the vertex $z_i$ of $K$ to vertices of
$P_2-\{x_1,x_2,x_3,x_4\}$; see Figure~\ref{fig:3pet5.30}.
To derive a
contradiction we first prove that
$\sigma(r_{11})\ne\sigma(r_{12})$. If
$\sigma(r_{11})=\sigma(r_{12})$, then
$\sigma(r_{21})=\sigma(r_{22})$ as well, because the outflow
from every nonempty set of vertices is $0$, by the Kirchhoff
law. This in turn implies that the sum of flow values on the
three edges incident with $z_1$ and different from $r_{11}$ and
$r_{12}$ must be $0$ as well. Similarly, the sum of flow values
on the three edges incident with $z_2$ and different from
$r_{21}$ and $r_{22}$ must be $0$. It follows that $\sigma$
induces a nowhere-zero $\mathbb{Z}_2\times\mathbb{Z}_2$-flow on
$P_1$, which is a contradiction. Therefore
$\sigma(r_{11})\ne\sigma(r_{12})$. By analogous arguments we
can show that $\sigma(r_{i1})\neq \sigma(r_{i2})$ for each
$i\in \{1,2,3,4\}$. Moreover, the fact that the outflow from
every nonempty set of vertices is $0$ also implies that
$\sigma(r_{11})+\sigma(r_{12})=\sigma(r_{21})+ \sigma(r_{22})$
and $\sigma(r_{31})+\sigma(r_{32})= \sigma(r_{41})+
\sigma(r_{42})$. Thus if we take the induced valuation on
$P_2-\{x_1x_2,x_3x_4\}$ and assign the value
$\sigma(r_{11})+\sigma(r_{12})$ to the edge $x_1x_2$ and  the
value  $\sigma(r_{31})+\sigma(r_{32})$ to the edge $x_3x_4$, we
obtain a nowhere-zero $\mathbb{Z}_2\times\mathbb{Z}_2$-flow on
$P_2$. This is again a contradiction, so $K$ admits no
nowhere-zero $\mathbb{Z}_2\times\mathbb{Z}_2$-flow.

If we substitute every vertex $s\in W$ in $K$ with a copy of a
cubic $5$-pole $M$, identifying the dangling edges of $M$ with
those of $K-s$ arbitrarily, we obtain a cubic graph $K_M$.
Since any nowhere-zero $\mathbb{Z}_2\times\mathbb{Z}_2$-flow on
$K_M$ would induce one on $K$, the graph $K_M$ is a snark.

For any fixed girth $g$ we create $M=M_g$ from a connected
bipartite cubic graph $L_g$ of girth $g=2m\ge 6$ by removing a
path of length $2$ and retaining the dangling edges. Such a
graph $L_g$ indeed exists: Theorem~4.8 in \cite{NS-largepw}
guarantees that there exists an arc-transitive cubic graph $X$
of girth $g$. If $X$ is bipartite, we can set $L_g=X$. If $X$
is not bipartite, for $L_g$ we can take its bipartite double
(the direct product $X\times K_2$ with the complete graph $K_2$
on two vertices), which is connected, cubic, bipartite, and has
girth $g$. Note that $L_g$ is also $3$-edge-colourable, because
it is bipartite.

Since $U$ is a decycling set of $K$, this choice of $M$ gives rise 
to a snark in which each cycle of length smaller than $g$ traverses 
either $u_1$ or $u_3$. We now create a
$(3,3)$-pole $F_g(S_1,S_2)=K_M-\{u_1,u_3\}$ where the
connectors are formed from the three dangling edges formerly
incident with $u_1$ and $u_3$, respectively. Clearly, $F_g$
contains no cycles of length smaller than $g$. Furthermore, the
fact that $K_M$ is a snark implies that every
$3$-edge-colouring of $F_g$ assigns colours $a,b,b$ and $a,c,c$
to the edges constituting $S_1$ and $S_2$, respectively, for
certain $a,b,c\in\mathbb{Z}_2\times\mathbb{Z}_2-\{0\}$, not
necessarily distinct. In~other words, the \emph{total flow}
through $F_g$, by which we mean the sum  $a+b+b=a+c+c=a$, is
always different from~$0$. Any dipole with this property is
called \emph{proper}.

Finally, we construct the required snark $\tilde G$ of girth
$g=2m$. We take the Petersen graph $P$ as the base graph $G$
for superposition and pick a $6$-cycle $C=(e_0e_1\ldots e_5)$
in it. Let $v_i$ denote the common vertex of the edges
$e_{i-1}$ and $e_i$ with indices taken modulo $6$; see
Figure~\ref{fig:core3}. The edges of $P$ not on $C$ form a
spanning tree $T$. For further reference, let $v$ denote the
central vertex of $T$ and let $u$ be the neighbour of $v$ that
is adjacent to $v_2$ and $v_5$. We substitute each of the edges
$e_1$, $e_2$, $e_4$, and $e_5$ with a copy of the $(3,3)$-pole
$F_g$ and each of the vertices $v_0$, $v_1$, $v_3$, and $v_4$
with a copy of the $5$-pole $M_g$. In addition, we substitute
both $v_2$ and $v_5$ with a copy of the $(3,3,1)$-pole $Z$
which consists of a single vertex $z$ incident with three
dangling edges and of two additional isolated edges; the edges
incident with the vertex contribute to all three connectors of
$Z$ while each isolated edge contributes to two different
connectors of size~$3$, see Figure~\ref{fig:superpos}. Finally,
we join the connectors of each copy of $F_g$ to a connector of
a copy of $M_g$ and a connector of a copy of $Z$, and connect
the copies of $M_g$ between themselves and to the remaining
vertices of $P$ in such a way that a cubic graph $\tilde G$
arises.

Next we prove that $\tilde G$ is a snark. If $\tilde G$ was
$3$-edge-colourable, then it would admit a nowhere-zero
$\mathbb{Z}_2\times\mathbb{Z}_2$-flow $\xi$. Recall that
$\tilde G$ is obtained from the Petersen graph $P$ by
substituting each of the edges $e_1$, $e_2$, $e_4$, and $e_5$
with a copy of the $(3,3)$-pole $F_g$. Now we define an edge
valuation $\xi_*$ on $P$ in
$\mathbb{Z}_2\times\mathbb{Z}_2-\{0\}$ as follows. For each
edge $e$ of $P$ different from the four previously mentioned
edges set $\xi_*(e)=\xi(e)$; for each of the remaining four
edges set $\xi_*(e)$ to be the total flow through the
corresponding copy of $F_g$. Since $\xi$ fulfils the Kirchhoff
law, so does $\xi_*$. Thus $\xi_*$ is a
$\mathbb{Z}_2\times\mathbb{Z}_2$-flow. However, $F_g$ is a
proper dipole, so $\xi_*$ is a nowhere-zero
$\mathbb{Z}_2\times\mathbb{Z}_2$-flow on the Petersen graph,
which is absurd. Therefore $\tilde G$ is a snark. (For a more
detailed argument see \cite[Theorem~4]{Ko}.)

By inspecting Figure~\ref{fig:3pet5.30} it is easy to see that
every edge cut in the graph $K$ has at least three edges.
Recall that $L_g$ is vertex-transitive of girth $g\ge 6$, and
therefore it is cyclically $g$-edge-connected by Theorem~17 of
\cite{NS-cc}. Since $M_g$ arises from $L_g$ by removing a path
of length $2$, every edge cut of $F_g$ has at least three
edges, too. Now, let us look at the cycle-separating edge cuts
in $\tilde{G}$. If such a cut disconnects at least two copies
of $F_g$, then it has at least six edges. If, on the other
hand, it disconnects exactly one copy of $F_g$, then the cut
contains at least three edges in the copy of $F_g$ and at least
two edges outside of $F_g$. Finally, if the cut does not
intersect any of the copies of $F_g$, then it has at least five
edges because the Petersen graph is cyclically
$5$-edge-connected, and each copy of $M_g$ is connected and is
separated from the rest by five edges. Summing up, every
cycle-separating edge cut in $\tilde{G}$ has at least five
edges, in other words, $\tilde G$ is cyclically
$5$-edge-connected. In particular, $\tilde G$ is a nontrivial
snark.

We further show that $\girth(\tilde G)=g$. If we take into
account the fact that the dipole $F_g$ contains no cycles of
length smaller than $g$ and that $\{v_0,v_1,v_3,v_4\}$ is a
decycling set of $P$, we can conclude that $\girth(\tilde G)\ge
g$. However, Theorem~4.8 in \cite{NS-largepw} states that there
exist infinitely many arc-transitive cubic graphs $X$ of any
given girth $g\ge 6$. It follows that such graphs may have
arbitrarily large diameter, and hence infinitely many of them
contain at least two disjoint $g$-cycles. Therefore $M_g$ can
be constructed in such a way that it still contains a
$g$-cycle. Summing up, $\girth(\tilde G) = g$. From
Proposition~\ref{prop:girth} we now infer that $\df{\tilde
G}\ge g/2$.

\begin{figure}[h!]
\centering
\includegraphics[width=0.65\textwidth]{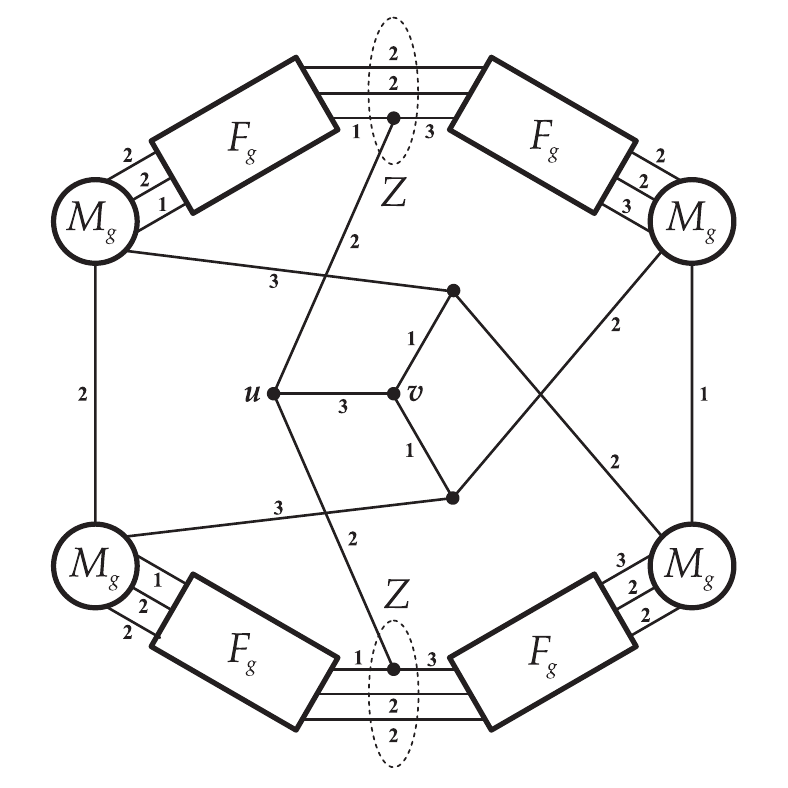}
\caption{The resulting snark $\tilde G$ of girth $g$}
\label{fig:superpos}
\end{figure}

Observe that our construction does not determine the snark
$\tilde G$ uniquely, because the order of semiedges in the
connectors is irrelevant for the result. We take this advantage
to show that the identification of the free ends of semiedges
in connectors can be performed in such a way that $\tilde
G-\{u,v\}$ is $3$-edge-colourable. For this purpose we first
extend the partial $3$-edge-colouring of $F_g$ shown in
Figure~\ref{fig:superedge} to the entire edge set. Recall that
$F_g$ was created from a bipartite cubic graph $L_g$ by
removing a path of length $2$, so $F_g$ is colourable. To make
the extension of the partial colouring possible we need to be
more specific about how the five dangling edges of $M_g$ are
joined to the five dangling edges of $K-s$ for every vertex
$s\in W$. To this end, it is sufficient to realise that by
Parity Lemma every $3$-edge-colouring of $M_g$ induces the
colour vector $aaabc$ where $a$, $b$, and $c$ are the three
nonzero elements of $\mathbb{Z}_2\times\mathbb{Z}_2$ in some
order. Although the edges receiving the lonely colours $b$ and
$c$ may not be chosen arbitrarily, we can always attach a
coloured copy of $M_g$ to $K-s$, possibly after permuting the
colours, in such a way that the colours of the corresponding
edges match. Hence, $F_g$ admits a $3$-edge-colouring where
each connector receives colours $1$, $1$, and $2$ as shown in
Figure~\ref{fig:superedge}. Finally, we insert the coloured
copies of $F_g$ and $M_g$ into $P$, possibly after permuting
the colours, in such a way that the colours of the edges in the
joined connectors again match. The result is a defective edge
colouring of $\tilde G$ where the Kirchhoff law fails only at
the vertices $u$ and $v$, see Figure~\ref{fig:superpos}. Thus
$\tilde  G-\{u,v\}$ is $3$-edge-colourable, and consequently,
the resistance of $\tilde G$ equals $2$. It follows that
$\omega(\tilde{G})=2$, as claimed. This completes the proof.
\end{proof}

The following interpretation of the previous proof is also
important, as one can see in our paper~\cite{KMNS-defect3}.

\begin{theorem}
There exist nontrivial snarks with arbitrarily large girth that
contain a non-removable pair of adjacent vertices.
\end{theorem}

Another benefit of the construction presented in the proof of
Theorem~\ref{thm:snarxlargegirth} is a strengthening of the
original Kochol's construction~\cite{Ko}.

\begin{theorem}
For every $g\geq 5$ there exist infinitely many cyclically
$5$-connected snarks whose girth equals $g$.
\end{theorem}
\begin{proof}
Snarks constructed in the proof of
Theorem~\ref{thm:snarxlargegirth} satisfy the statement for
every even $g\geq 6$. If $g\geq 7$ is odd, we modify the
construction by taking $L_g$ from the infinitely many graphs of
girth $g$ constructed in Theorem~4.8 in \cite{NS-largepw}.
Since we do not care whether $L_g$ is $3$-edge-colourable or
not, we do not require $L_g$ to be bipartite. Otherwise, the
construction proceeds as in the proof of
Theorem~\ref{thm:snarxlargegirth}. Finally, if $g=5$, there are
several available constructions of infinitely many cyclically
$5$-connected snarks of girth $5$, for example rotation snarks
or permutation snarks constructed in Theorem~5.1 and
Example~6.4 of \cite{MS-superp}, respectively.
\end{proof}

It is also possible to construct snarks with cyclic
connectivity $4$ and girth $g$ for each $g\ge 5$. The
construction is similar to that described in the proof of
Theorem~\ref{thm:snarxlargegirth} except that one has to use
the method of Section~5 of \cite{Ko} instead of Section~4.

\section{Final remarks}
\noindent{}We believe that nontrivial snarks with any given
oddness $\omega$ and colouring defect $d$ exist for each pair
$d\ge 3\omega/2$ (which is the restriction posed by
Theorem~\ref{thm:oddness}). Constructing such snarks would be
worthwhile as it would provide a complete generalisation of
Theorem~3.4 of Jin and Steffen~\cite{JS} to nontrivial snarks.
Unfortunately, our construction, which only deals with
$\omega=2$, does not easily generalise to larger values of
$\omega$.

Another possibility is to consider an analogous (but weaker)
problem where oddness is replaced with resistance. Recall that
$\omega(G)\ge\rho(G)$ for every bridgeless cubic graph $G$. It
follows from Theorem~\ref{thm:oddness} that $\df{G}\geq
3\rho(G)/2$, and we may ask whether for any given $\rho\ge 2$
and $d\ge 3\omega/2$ there exists a nontrivial snark with
resistance $\rho$ and colouring defect $d$. We think that the
answer is ``yes''.

\section*{Acknowledgements}
\noindent{}This research was partially supported by the grant
No.~APVV-19-0308 of Slovak Research and Development Agency. The
first and the third author were partially supported by the
grant VEGA~2/0078/20 of Slovak Ministry of Education. The
second author and the fourth author were partially supported by
the grant VEGA-1/0743/21. The authors would like to express
their gratitude to an anonymous referee whose suggestions
helped to improve the presentation of this paper.

\end{document}